\documentclass[12pt,reqno]{amsart}
\usepackage[utf8]{inputenc}
\usepackage{amsmath} 
\usepackage{amsthm} 
\usepackage{amssymb} 
\usepackage{enumitem} 
\usepackage{mathtools} 
\usepackage{mathrsfs} 
\usepackage{hyperref} 

\newtheorem{theorem}{Theorem}[section]
\newtheorem{corollary}[theorem]{Corollary}

\newtheorem{lemma}[theorem]{Lemma}
\newtheorem{proposition}[theorem]{Proposition}

\theoremstyle{definition}
\newtheorem{definition}[theorem]{Definition}
\newtheorem{remark}[theorem]{Remark}
\newtheorem*{notation}{Notation}
\newcommand{\ep}{\varepsilon}

\newcommand{\dd}{\mathop{}\!\mathrm{d}}
\newcommand{\suchthat}{\,:\,}

\newcommand{\mres}{\mathbin{\vrule height 1.6ex depth 0pt width
0.13ex\vrule height 0.13ex depth 0pt width 1.3ex}}
\newcommand{\test}{\varphi}
\newcommand{\opa}[1]{\mathcal{A}\left[ #1 \right]}
\newcommand{\opb}[1]{\mathcal{B}\left[ #1 \right]}

\newcommand{\opd}[1]{\mathcal{D}\left[ #1 \right]}
\newcommand{\opastar}[1]{\mathcal{A}^{*}\left[ #1 \right]}
\newcommand{\opbstar}[1]{\mathcal{B}^{*}\left[ #1 \right]}
\newcommand{\reacc}{\mathcal{G}}
\newcommand{\coreacc}{\mathcal{C}}
\newcommand{\entropy}{\mathcal{H}}
\newcommand{\energy}[1]{\mathcal{E}\left[ #1 \right]}
\newcommand{\initialvalue}{\mu_{\text{\normalfont in}}}
\newcommand{\initialvalueac}{f_{\text{\normalfont in}}}
\newcommand{\fixedpoint}{\mu}
\newcommand{\fixedpos}{\bar{\mu}}
\newcommand{\sizebound}{X_0}
\newcommand{\sizefunction}{q_{\delta}}
\newcommand{\chifunction}[1]{\chi_{\raisebox{-.5ex}{$I_{#1,\delta}$}}}

\newcommand{\naturals}{\mathbb{N}}

\newcommand{\real}{\mathbb{R}}

\newcommand{\measure}{\mathcal{M}_{b}(0,1]}
\newcommand{\measureprod}{\mathcal{M}_{b}\left((0,1]^2\right)}
\newcommand{\measurepos}{\mathcal{M}^{+}_{b}(0,1]}
\newcommand{\measureposbdy}{\mathcal{M}^{+}_{b}(1,\infty)}
\newcommand{\weakmeasure}{\mathcal{M}^{\text{w}\star}_{b}(0,1]}

\newcommand{\solspaceinfpos}{C\big([0,\infty); \mathcal{M}^{+}_{b}(0,1]\big)}
\newcommand{\weaksolspace}{C\big([0,T]; \mathcal{M}^{\text{\normalfont w}\star}_{b}(0,1]\big)}

\newcommand{\vanishing}{C_0(0,1]}
\newcommand{\bounded}{C_b(0,1]}

\newcommand{\bddborel}{\mathcal{B}_b(0,1]}

\newcommand{\lip}{\mathrm{Lip}_0(0,1]}
\newcommand{\abs}[1]{\left\lvert#1\right\rvert}
\newcommand{\norm}[1]{\left\lVert#1\right\rVert}
\newcommand{\variation}[1]{\| #1 \|_{\text{\normalfont TV}}}
\newcommand{\lipnorm}[1]{[#1]_{\text{\normalfont Lip}}}
\newcommand{\supnorm}[1]{\left\lVert#1\right\rVert_{\infty}}

\newcommand{\compactsuppx}{C_c(X)}
\newcommand{\vanishingx}{C_0(X)}
\newcommand{\boundedx}{C_b(X)}

\newcommand{\bddborelx}{\mathcal{B}_b(X)}

\newcommand{\measurex}{\mathcal{M}_{b}(X)}
\newcommand{\measureposx}{\mathcal{M}^{+}_{b}(X)}
\newcommand{\weakmeasurex}{\mathcal{M}^{\text{w}\star}_{b}(X)}

\newcommand{\upperbound}{K_0}
\newcommand{\lowerbound}{K_1}
\newcommand{\upfrag}{F_0}
\newcommand{\kbound}{K_{\infty}}
\newcommand{\powerlaw}{K_{\beta}}
\newcommand{\lowerboundbdy}{\theta_{\beta}}
\newcommand{\approxkernel}{K_{j}}
\newcommand{\approxsol}{\mu^{j}}
\newcommand{\approxsolsub}{\mu^{j_k}}
\newcommand{\approxreacc}{\mathcal{G}^{j}}
\newcommand{\mass}{\mathfrak{m}}
\newcommand{\massbdy}{\mathfrak{M}}
\newcommand{\momentbound}{M}
\newcommand{\integrate}[2]{\langle #1, #2 \rangle}
\newcommand{\weak}{\stackrel{\star}{\rightharpoonup}}
\newcommand{\bweak}{\stackrel{b}{\rightharpoonup}}

\newcommand{\lebesgue}{\mathcal{L}}

\numberwithin{equation}{section}

\title[Coagulation equations with boundary data]{Well-posedness for boundary value problems for coagulation-fragmentation equations}

\author{I\~{n}igo U. Erneta}
\address{I. U. Erneta:
	Universitat Polit\`ecnica de Catalunya and BGSMath, Departament de Matem\`{a}tiques, Diagonal 647, 08028 Barcelona, Spain}
\email{inigo.urtiaga@upc.edu}


\begin{document}
	
\begin{abstract}
We investigate a coagulation-fragmentation equation with boundary data, establishing the well-posedness of the initial value problem when the coagulation kernels are bounded at zero and showing existence of solutions for the singular kernels relevant in the applications.
We determine the large time asymptotic behavior of solutions,
proving that solutions converge exponentially fast to zero in the absence of fragmentation and stabilize toward an equilibrium if the boundary value satisfies a detailed balance condition.
Incidentally, we obtain an improvement in the regularity of solutions by showing the finiteness of negative moments for positive time.
\end{abstract}	


\maketitle
\section{Introduction}

Coagulation and fragmentation processes occur in large systems of particles where the constituents have a tendency to form clusters of matter.
This phenomenon has been widely observed in a variety of scientific areas, such as in atmospheric science (growth of aerosol clusters in the atmosphere \cite{drake1972general}), chemistry (study of colloids \cite{von1916drei}, formation of polymers \cite{stockmayer1943theory}), medicine (hematology \cite{perelson1984kinetics}) and astrophysics (formation of galaxies \cite{safronov1972evolution}).

A first mathematical model to describe coagulation was proposed by Smoluchowski \cite{von1916drei, smoluchowski1918versuch} which was later upgraded by Melzak to include the effect of
fragmentation \cite{melzak1957scalar}.
Assuming that the coagulation and fragmentation processes are entirely characterized by the volume (\emph{size}) of the clusters, the evolution of the cluster concentration is determined by an infinite system of ODEs, the (discrete) \emph{coagulation-fragmentation equation}
\begin{equation}\label{discrete_smoluchowski}
\begin{split}
\dfrac{\dd c_i(t)}{\dd t} &= \dfrac{1}{2}\sum_{j = 1}^{i - 1} K_{j, i - j} c_{j}(t) c_{i - j}(t) - \sum_{j = 0}^{\infty} K_{i, j} c_{i}(t)c_{j}(t) \\
&\quad - \dfrac{1}{2}\sum_{j = 1}^{i - 1} F_{j, i - j} c_{i}(t) + \sum_{j = 0}^{\infty} F_{i, j} c_{i + j}(t),
\end{split}
\end{equation}
where $c_i(t)$ denotes the concentration of clusters of size $i = 1, 2, \cdots$ at time $t$ and the coefficients $K_{i, j}$ and $F_{i, j}$ stand for the coagulation and fragmentation rates, respectively.

Recent investigations in atmospheric chemistry  \cite{mcgrath2012atmospheric, olenius2013free}
have led to the consideration of a truncated version of \eqref{discrete_smoluchowski}, namely
\begin{equation}\label{birth_death}
\begin{split}
\dfrac{\dd c_i(t)}{\dd t} &= \dfrac{1}{2}\sum_{j = 1}^{i - 1} K_{j, i - j} c_{j}(t) c_{i - j}(t) - \sum_{j = 0}^{N} K_{i, j} c_{i}(t)c_{j}(t) \\
&\quad - \dfrac{1}{2}\sum_{j = 1}^{i - 1} F_{j, i - j} c_{i}(t) + \sum_{j = 0}^{N} F_{i, j} c_{i + j}(t) + Q_i - S_i,
\end{split}
\end{equation}
where $i = 1, 2, \cdots, N$ for some $N\in \naturals$
and $Q_i,$ $S_i$ are additional source and sink terms
accounting for the interaction with larger clusters and the effect of supplementary mechanisms.
\footnote{Under the assumption that particles of size smaller than $N$ may merge with larger ones, the sink term $S_i$ could in principle depend on $c_i$ as in the case of binary coagulation where these quantities would be proportional.}
Solving \eqref{birth_death} corresponds to finding the concentration of \emph{small} clusters ($i \leq N$) as a function of time assuming that the concentration of the \emph{large} ones ($i > N$) is a given datum.
Thus, equation \eqref{birth_death} may be interpreted as a boundary value problem for the coagulation--fragmentation equation.

The motivation for the boundary value problem above arises from the study of the dynamics of aerosols in the atmosphere.
In the description of these systems there appear multiple physical phenomena that affect particles in a different manner according to their size
and which further act at different time scales.
The model described by a coagulation equation
together with prescribed values of the large cluster distribution
would yield the evolution of cluster sizes
in systems where the larger particles are not determined by coagulation but by a faster process.

For instance, as indicated in \cite{friedlander2000smoke}, Chapter 13 the size distribution of
particles in the atmosphere with diameter larger than 2.5 $\mu$m
(the \emph{coarse mode})
is due to mechanical processes such as wind, sedimentation and turbulent mixing, and not to the aggregation of smaller particles.
It is then natural to assume that the concentration of these larger particles is given and acts as a boundary value term for the smaller particles.

In this work, we study a continuous version of \eqref{birth_death}
and investigate the well-posedness of the associated Cauchy problem.
The continuous coagulation fragmentation equation corresponds to the integro-differential equation
\begin{equation}\label{coagulation_fragmentation_equation}
\begin{split}
\partial_t f(t, x) &= \frac{1}{2} \int_{0}^{x} K(x - y, y) f(t, x-y) f(t, y)\dd y - \int_{0}^{\infty} K(x, y) f(t, x) f(t, y)\dd y\\
&\quad-\frac{1}{2} \int_{0}^{x} F(x - y, y) f(t, x)\dd y + \int_{0}^{\infty} F(x, y) f(t, x + y)\dd y
\end{split}
\end{equation}
where the kernels $K,$ $F$ are nonnegative symmetric functions, the \emph{coagulation} and \emph{fragmentation kernels}, and $f(t, x)$ is the size distribution function for volumes
$x$ in $(0, \infty)$.
Assuming that the restriction of $f(t, \cdot)$ to large clusters $\{x > \sizebound\}$
is a given \textit{boundary datum} $g(t, x)$ for $x \in (\sizebound, \infty)$ the above equation turns into
\begin{equation}\label{bdy_coagulation_fragmentation_equation}
\begin{split}
\partial_t f(t, x) &= \frac{1}{2} \int_{0}^{x} K(x - y, y) f(t, x-y) f(t, y)\dd y - \int_{0}^{\sizebound} K(x, y) f(t, x) f(t, y)\dd y \\
&\quad- \int_{\sizebound}^{\infty} K(x, y) f(t, x) g(t, y)\dd y + \int_{\sizebound - x}^{\infty} F(x, y) g(t, x + y)\dd y\\
&\quad-\frac{1}{2} \int_{0}^{x} F(x - y, y) f(t, x)\dd y + \int_{0}^{\sizebound - x} F(x, y) f(t, x + y)\dd y
\end{split}
\end{equation}
for $x \in (0, \sizebound),$
the \emph{boundary valued coagulation-fragmentation equation}.

A detailed review of available well-posedness and long time asymptotic properties for the solutions of coagulation-fragmentation models in the whole real line can be found in  \cite{banasiak2019analytic}.
To our knowledge, well-posedness results for coagulation-fragmentation equations with boundary data have not been considered in the mathematical literature.

\subsection{Framework}
We are interested in continuous kernels in $(0, \infty)^2$ satisfying
\begin{equation}\label{upper_bound_singular_kernel}
 K(x, y) \leq \upperbound (x^{-\alpha} + y^{-\alpha})(x^{\beta} + y^{\beta})
\end{equation}
\begin{equation}\label{fragmentation_upper_bound}
F(x, y) \leq \upfrag (x^{\gamma} + y^{\gamma})
\end{equation}
for some $\upperbound,$ $\upfrag \geq 0$ and $\alpha, \beta, \gamma \in [0,1].$
This assumption includes most of the singular coagulation kernels in the applications \cite{aldous1999deterministic}.

By scaling we may assume without loss of generality that $\sizebound = 1$.
Multiplying \eqref{bdy_coagulation_fragmentation_equation} by a
test function
$\test \in C^{\infty}_c(0, 1)$
and formally integrating
\begin{equation}\label{bdy_coagulation_fragmentation_ac}
\begin{split}
\int_{0}^{1}\partial_t f(t, x) \test(x) \dd x + \int_{0}^{1} f(t, x) \big(\reacc(t, x)\test(x) + \opb{\test}(x) \big) \dd x =\\
= \dfrac{1}{2}\int_{0}^{1}\int_{0}^{1} K(x, y) f(t, x) f(t, y) \opa{\test}(x, y) \dd x \dd y + \int_{0}^{1} \coreacc(t, x)\test(x) \dd x
\end{split}
\end{equation}
where we have introduced the operators
\begin{equation*}
\opa{\test}(x, y) = \test(x + y) - \test(x) - \test(y)
\end{equation*}
\begin{equation*}
\opb{\test}(x) = \dfrac{1}{2}\int_{0}^{x} F(x - y, y) (\test(x) - \test(x - y) - \test(y)) \dd y,
\end{equation*}
and the functions
\begin{equation*}
\reacc(t, x) = \int_{1}^{\infty} K(x, y) g(t, y) \dd y, \quad \coreacc(t, x) = \int_{1}^{\infty} F(y - x, x) g(t, y) \dd y.
\end{equation*}

It will be convenient to rewrite equation \eqref{bdy_coagulation_fragmentation_ac} for general measure-valued solutions, hence we now fix some relevant notation:

\begin{notation}
Given a topological space $X,$ the usual spaces of continuous functions are denoted by $\compactsuppx$ (compactly supported), $\vanishingx$ (vanishing at infinity i.e. the completion of $\compactsuppx$ in the supremum norm) and $\boundedx$ (bounded).
We also write $\bddborelx$ for the set of bounded Borel functions.

Given a function $\test \colon X \to \real$ we denote its Lipschitz seminorm by
\[\lipnorm{\test} = \sup_{x, y \in X,\, x \neq y} \dfrac{|\test(x) - \test(y)|}{|x-y|}.\]

The space of signed Borel measures with bounded total variation is denoted by $\measurex$ and we write $\measureposx$ if the measures are nonnegative (i.e. with $\mu = \abs{\mu}$).

In this work we will only have $X = (0, 1]$ or $(0, 1]^2$ with the topology induced by $(0, \infty)$ and $(0, \infty)^2,$ respectively.
In particular, $X$ will always be a Polish space and the measures defined thereon will have some useful properties for constructing solutions.

Assume $X$ to be Polish.
Writing $\variation{\mu} = \abs{\mu}(X)$ for the total variation norm, recall that we have a norm preserving isomorphism of Banach spaces
\[(\measurex, \variation{\cdot}) \cong ((\vanishingx)', \|\cdot\|_{\text{op}}), \quad \mu \mapsto (\test \mapsto \textstyle\int \test \dd \mu)\]
where $(\vanishingx)'$ is the dual of $\vanishingx$ and $\|\cdot\|_{\text{op}}$ the operator norm.
Hence $\variation{\mu} = \sup \{ \int \test \dd \mu \suchthat \test \in \vanishingx\}$ and we will always employ the notation for the dual pairing $\integrate{\mu}{\test} = \int \test \dd \mu.$
In particular, we may endow $\measurex$ with the weak-$\star$ topology of $(\vanishingx)'$ (the \emph{weak-$\star$ topology}) or that of $(\boundedx)'$ (the \emph{weak-$C_b$ topology}).

Weak-$\star$ (resp. weak-$C_b$) convergence is denoted by $\weak$ (resp. $\bweak$).
By default the spaces $\measurex$ and $\measureposx$ are assumed to carry the topology generated by $\variation{\cdot}.$
We write $\weakmeasurex$ for the set $\measurex$ with the weak-$\star$ topology.
For a more complete exposition of the subject of spaces of measures we refer the reader to \cite{bogachev2007measure}.

We often write $\reacc_t, \coreacc_t$ for the maps $t \mapsto \reacc(t, \cdot),$ $t \mapsto \coreacc(t, \cdot).$

Finally, for functions $f, g \colon [0, \infty) \to [0, \infty)$ we will say that $f = \mathcal{O}(g)$ as $t \to \infty$ if there are nonnegative constants $t_0, C $ such that $f(t) \leq C g(t)$ for $t \geq t_0.$
\end{notation}

Given an initial value $\initialvalue \in \measurepos$ and a map $\mu \colon [0, \infty) \to \measurepos,$ $t \mapsto \mu_t,$
if we let
$f(t, x) \dd x \to \dd \mu_t(x)$ in \eqref{bdy_coagulation_fragmentation_ac}
we obtain
a weak formulation of the Cauchy problem for the boundary valued coagulation fragmentation equation
\begin{equation}\label{cauchy_pb}
\begin{cases}
\dfrac{\dd}{\dd t} \integrate{\mu_t}{\test} + \integrate{\mu_t}{\reacc_t \test + \opb{\test}} = \dfrac{1}{2}\integrate{\mu_t \otimes \mu_t}{ K \opa{\test}} + \integrate{\lebesgue}{\coreacc_t\test}, \text{ for } t \in (0, \infty)\\
\mu_0 = \initialvalue
\end{cases}
\end{equation}
where the test functions $\test$ are in $C^{\infty}_c(0,1)$ and $\lebesgue$ is the Lebesgue measure restricted to $(0, 1].$

\begin{definition}\label{boundary_datum}
A \emph{boundary datum} is a nonnegative continuous map
\[g \colon [0, \infty) \to L^{1}(1, \infty), \quad t \mapsto g(t, \cdot).\]
For $\lambda \in \real$, the corresponding \emph{moment of order $\lambda$} is the function of time
\[\massbdy_{\lambda}(t) = \int_{1}^{\infty} x^{\lambda} g(t, x) \dd x\]
and the uniform bounds are denoted by
\[\momentbound_{\lambda} = \sup_{t \geq 0} \massbdy_{\lambda}(t).\]
Similarly, we denote the moments of measures $\mu \colon [0, \infty) \to \measurepos$ by
\[\mass_{\lambda}(t) = \int_{(0,1]} x^{\lambda} \dd \mu_t(x).\]
\end{definition}

\begin{remark}
The moments need not be finite for general data.
\end{remark}

\begin{remark}
We will only make assumptions on the moments of the boundary datum, therefore our analysis applies equally well to more general continuous maps
\[ g \colon [0, \infty) \to \measureposbdy, \text{ with } \quad \massbdy_{\lambda}(t) = \int_{1}^{\infty} x^{\lambda} \dd g_t(x).\]
\end{remark}

We introduce the space of regular functions
\begin{equation*}
\lip = \left\{ \test: (0,1]\to \real\suchthat \test \text{ is Lipschitz and vanishes at }0 \right\}
\end{equation*}
and extend the operator $\mathcal{A}$ to functions having values at $1$ by letting
\begin{equation*}
\opa{\test}(x, y) = \overline{\test}(x + y) - \test(x) - \test(y) \quad \text{ for } x, y \in (0,1]
\end{equation*}
where $\overline{\test}$ denotes the extension to $(0, \infty)$ by a constant
\begin{equation*}
\overline{\test}(x) =
\begin{cases}
\test(x), & \text{ if } x \in (0, 1]\\
\test(1), & \text{ if } x \in (1, \infty).
\end{cases}
\end{equation*}
From now on $\mathcal{A}$ will always denote this extension.

Test functions in $\lip$ cancel the singularities of the kernels:

\begin{lemma}\label{uniform_bounds_operator}
If the kernels satisfy \eqref{upper_bound_singular_kernel} and \eqref{fragmentation_upper_bound}, then for $\test\in\lip$ we have
\[\sup_{x, y \in (0,1]} K(x, y) \big|\opa{\test}(x,y)\big| \leq 8 \upperbound \lipnorm{\test}, \quad \sup_{x \in (0, 1]} \abs{\opb{\test}(x)} \leq 3 \upfrag \supnorm{\test},\]
\[\sup_{x \in (0, 1]} \reacc(t, x) \abs{\test(x)} \leq 4 \upperbound \massbdy_{\beta}(t) \lipnorm{\test}, \quad \sup_{x \in (0, 1]} \coreacc(t, x) \leq 2 \upfrag \massbdy_{\gamma}(t).\]
\end{lemma}
\begin{proof}
The extension $\overline{\test}$ is again Lipschitz with $\lipnorm{\overline{\test}} = \lipnorm{\test}.$
By symmetry
\begin{equation*}
\abs{\opa{\test}(x, y)} \leq 2 \lipnorm{\test} \min\{x, y\}
\end{equation*}
and the first inequality follows by expanding in \eqref{upper_bound_singular_kernel}.

The second inequality is clear,
as $F$ is bounded by $2 \upfrag$ and $x\leq 1$.

\eqref{upper_bound_singular_kernel} implies
$\reacc(t, x) \leq 4 \upperbound x^{-\alpha} \massbdy_{\beta}(t)$
and since $\abs{\test(x)} \leq \lipnorm{\test} x$ and $\alpha \leq 1$
the third claim follows.

Finally, since $\gamma \geq 0$ we have $\coreacc(t, x) \leq \upfrag \int_{1}^{\infty} (1 + y^{\gamma}) g(t, y) \dd y \leq 2 \upfrag \massbdy_{\gamma}(t).$
\end{proof}

Lemma \ref{uniform_bounds_operator} now justifies that we formulate Problem \eqref{cauchy_pb} in the following manner.

\begin{definition}\label{definition_of_solution}
Given $\initialvalue \in \measurepos$ and a boundary datum $g,$ we say that a map $\mu \in \solspaceinfpos$ solves the
boundary valued coagulation-fragmentation equation if
\begin{equation}\label{a_solution}
\begin{split}
\integrate{\mu_t}{\test} = \integrate{\initialvalue}{\test} + \int_{0}^{t}
\Big(\dfrac{1}{2}\integrate{\mu_s \otimes \mu_s}{ K \opa{\test}} - \integrate{\mu_s}{\reacc_s \test + \opb{\test}} + \integrate{\lebesgue}{\coreacc_s \test}
\Big)\dd s
\end{split}
\end{equation}
for all $\test \in \lip,$ $t \geq 0.$
\end{definition}

\subsection{Main results}

We have the following existence result for the Cauchy problem when the kernels are singular.
\begin{theorem}\label{existence_result}
Let $K$, $F$ be continuous kernels satisfying \eqref{upper_bound_singular_kernel}, \eqref{fragmentation_upper_bound}, $\initialvalue$ an initial value in $\measurepos$ and $g$ a boundary datum with $\momentbound_{\max{\{\beta, \gamma\}}} < \infty.$
Then there exists a solution $\mu$ of the associated boundary valued coagulation-fragmentation equation.
\end{theorem}

\begin{remark}
Our methods do not yield the uniqueness of solutions, which is not clear under our mild assumptions on the initial datum as it need not have finite negative moments.
For uniqueness results for coagulation--fragmentation equations with singular kernels see \cite{camejo2015regular}, \cite{saha2015singular}.
\end{remark}

Under the assumption of lower bounds for the singular coagulation kernel
\begin{equation}\label{lower_bound_kernel}
K(x, y) \geq \lowerbound (x^{-\alpha} y^{\beta} + y^{-\alpha} x^{\beta})
\end{equation}
we can improve the regularity of the solutions.
We have
\begin{theorem}\label{improvement_result}
If the coagulation kernel satisfies \eqref{lower_bound_kernel}, then for all $\varepsilon > 0$ solutions of the boundary valued coagulation equation have $\mass_{-(1 + \alpha - \varepsilon)} \in L^{1}_{\text{loc}}[0, \infty),$ as well as $\mass_{-(1 - \varepsilon)}(t) < \infty$ for all $t > 0$ and equation \eqref{a_solution} holds for general test functions $\test$ in $\bddborel$ also.
\end{theorem}

In the absence of fragmentation, the \emph{boundary valued coagulation equation} reads
\begin{equation}\label{coagulation_alone}
\dfrac{\dd}{\dd t} \integrate{\mu_t}{\test} + \integrate{\mu_t}{\reacc_t \test} = \dfrac{1}{2}\integrate{\mu_t \otimes \mu_t}{ K \opa{\test}}, \quad \text{ for } t \in (0, \infty).
\end{equation}
The large time asymptotics of the solutions in this case
are determined by the following
\begin{theorem}\label{moment_asymptotics_result}
If the coagulation kernel satisfies \eqref{lower_bound_kernel}, then for any $\lambda \in \real$ solutions of the boundary valued coagulation equation (i.e. with $F = 0$)
have $\mass_{\lambda}(t) < \infty$ for $t > 0$ and
\begin{equation}
\mass_{\lambda}(t) = \mathcal{O}\bigg(\exp{\Big(-\lowerbound \int_{1}^{t} \massbdy_{\beta}(s) \dd s \Big)}\bigg) \quad \text{ as } t \to \infty.
\end{equation}
In particular,
\begin{itemize}
\item if $\massbdy_{\beta}(s) \geq \lowerboundbdy$ then $\variation{\mu_t} = \mathcal{O}(e^{-a t}) \quad \text{ as } t \to \infty$
\item if $\massbdy_{\beta}(s) \geq \lowerboundbdy t^{-1}$ then $\variation{\mu_t} = \mathcal{O}(t^{-a}) \quad \text{ as } t \to \infty$
\end{itemize}
where $a = \lowerbound \lowerboundbdy.$
\end{theorem}

Since $\initialvalue$ need not have finite negative moments,
their finiteness for positive times is a regularizing effect of the equation.

If coagulation and fragmentation are both present,
it is not clear \emph{a priori} whether solutions should converge to an equilibrium.
When the initial datum has a density with respect to the Lebesgue measure (written $\initialvalue \ll \lebesgue$ from now on)
we will show that the solution has the same property after restricting $\mu_t$ to $(0, 1)$ (written $\mu_t \mres (0,1)$) and we recover the strong formulation in equation \eqref{bdy_coagulation_fragmentation_equation}.
We study the large time asymptotics in this case.

Assuming a time independent boundary datum $g$, a stationary solution $f_{\infty}$ of \eqref{bdy_coagulation_fragmentation_equation} should solve
\begin{equation}\label{stationary_with_h}
\begin{split}
0 = & \frac{1}{2}\int_{0}^{x} \Big( K(x - y, y) f_{\infty}(x - y)f_{\infty}( y) - F(x - y, y) f_{\infty}(x) \Big)\dd y \\
&- \int_{0}^{1 - x}\Big( K(x, y) f_{\infty}(x) f_{\infty}(y) - F(x , y) f_{\infty}(x + y)\Big) \dd y\\
&- \int_{1 - x}^{1} \Big( K(x, y) f_{\infty}(x) f_{\infty}(y) - F(x , y) g(x + y)\Big) \dd y\\
&- \int_{1}^{\infty} \Big( K(x, y) f_{\infty}(x) g(y) - F(x , y) g(x + y)\Big) \dd y
\end{split}
\end{equation}
which is certainly satisfied if the function
\begin{equation}
Q(x) =
\begin{cases}
f_{\infty}(x), \quad \text{ if } x \in (0, 1)\\
g(x), \quad \text{ if } x \in (1, \infty).
\end{cases}
\end{equation}
is such that the \emph{detailed balance} condition
\begin{equation}\label{detailed_balance}
K(x, y) Q(x) Q(y) = F(x, y) Q(x + y)
\end{equation}
holds for $x, y \in (0, \infty).$

In particular,
if we have $K(x, y)g(y) > 0$ for $x \in (0, 1), y \in (1,\infty),$
then the profile $f_{\infty}$ is already determined by the boundary datum via
\begin{equation}\label{definition_f}
f_{\infty}(x) = f_{\infty}(g; x) = \dfrac{F(x, y)}{K(x, y)}\dfrac{g(x + y)}{g(y)}.
\end{equation}

\begin{definition}\label{bdy_detailed_balance}
We say that a boundary datum $g$ satisfies the \emph{detailed balance condition} if $f_{\infty}(g; \cdot)$ above is well defined and the corresponding function $Q = Q(g)$ satisfies \eqref{detailed_balance}.
\end{definition}

Under these special conditions, the analysis of Lauren{\c{c}}ot and Mischler \cite{laurenccot2002continuous}
(see also Ca\~{n}izo \cite{canizo2007convergence} for the discrete case)
shows that every solution must converge to a unique equilibrium.
More concretely, we have
\begin{theorem}\label{asymptotics_detailed_balance_result}
Let $K$ be a \emph{bounded coagulation kernel} (as defined below in Section \ref{existence}), $F$ satisfying $\eqref{fragmentation_upper_bound},$ $\initialvalueac \in L^1(0, 1)$ an initial condition such that $\initialvalueac \log{\initialvalueac} \in L^1(0, 1)$ and $g$ a time independent boundary datum with $\momentbound_{\max{\{\beta, \gamma\}}} < \infty$
satisfying the detailed balance condition.
If $f$ is a solution of the boundary valued coagulation fragmentation equation, then
\begin{equation*}
f(t, \cdot) \to f_{\infty}(g; \cdot) \text{ in the weak } L^1(0, 1) \text{ topology as } t \to \infty.
\end{equation*}
\end{theorem}

The paper is organized as follows.
Section \ref{existence} is devoted to the proofs of theorems \ref{existence_result} and \ref{improvement_result}.
The former
is obtained by a fixed point argument for bounded kernels together with suitable \emph{a priori} estimates,
while the latter
results from a careful study of the moments of measure solutions.
In section \ref{large_time_behavior}
we prove theorem \ref{moment_asymptotics_result}
by applying a similar analysis to the moments when the fragmentation term is absent.
We also show theorem \ref{asymptotics_detailed_balance_result} by finding bounds for an entropy function and extracting a subsequence converging to a stationary solution.

\section{Existence of solutions}\label{existence}

We will prove Theorem \eqref{existence_result} by approximating the singular coagulation kernels via kernels which are bounded in a vicinity of zero and extracting a convergent subsequence from the corresponding solutions.
\subsection{Well-posedness for bounded kernels}
\begin{definition}\label{kernel}
A nonnegative symmetric Borel function $K\colon (0, \infty)\times (0, \infty) \to \real$ is a \emph{bounded coagulation kernel} if it satisfies the following conditions
\begin{enumerate}
\item $\displaystyle\sup_{x, y \in (0,1]} K(x, y) < \infty$
\item there is a $\powerlaw > 0$ such that $\sup_{x \in (0,1]} K(x, y) \leq \powerlaw \, y^{\beta}$ for all $y \in (1, \infty).$
\end{enumerate}
\end{definition}

\begin{remark}
Despite our terminology, the \emph{bounded} kernels above are only bounded in the square region $(0,1]^2$ but may be unbounded outside.
The second condition ensures that the growth of the kernel as the volumes become large is reasonably controlled.
\end{remark}

For bounded kernels a fixed point argument shows that we have existence and uniqueness of solutions for the Cauchy problem.
Our proof is based on that of Norris \cite{norris1999smoluchowski}.
\begin{proposition}\label{wp_bdd_frag}
Let $K$ be a bounded coagulation kernel, $F$ satisfying \eqref{fragmentation_upper_bound}, $\initialvalue$ in $\measurepos$ and $g$ a boundary datum with $M_{\max{\{\beta, \gamma\}}} < \infty.$
Then there is a unique $\mu$ in $\solspaceinfpos$ such that
\[\begin{split}
\integrate{\mu_t}{\test} = \integrate{\initialvalue}{\test} + \int_{0}^{t}
\Big(\dfrac{1}{2}\integrate{\mu_s \otimes \mu_s}{ K \opa{\test}} - \integrate{\mu_s}{\reacc_s \test + \opb{\test}} + \integrate{\lebesgue}{\coreacc_s \test}
\Big)\dd s
\end{split}\]
for all $\test$ in $\bddborel.$
\end{proposition}
\begin{proof}
By equation \eqref{fragmentation_upper_bound} we have that
\[\reacc(t, x) \leq \powerlaw \momentbound_{\beta}, \quad \coreacc(t, x) \leq 2 \upfrag \momentbound_{\gamma}, \quad \sup_{x, y \in (0,1]} F(x, y) \leq 2 \upfrag
\]
and we write $\displaystyle \kbound = \sup_{x, y \in (0,1]} K(x, y).$

Fix $T > 0.$
We let
\begin{equation*}
R = \variation{\initialvalue} + \dfrac{\upfrag }{1 + \gamma}\big(\variation{\initialvalue} + 2 \upfrag \momentbound_{\gamma} T \big) T +  2 \upfrag \momentbound_{\gamma}
\end{equation*}
and choose $\tau$ with $0 < \tau \leq T$ such that
\begin{equation*}
\begin{split}
\big(6 \kbound R^2  + 2 \big(\powerlaw \momentbound_{\beta} + 3 \upfrag\big) R + 2 \upfrag \momentbound_{\gamma}\big) \tau \leq R \\
\left(6 \kbound + \powerlaw \momentbound_{\beta} + 3 \upfrag \right) \tau < 1.
\end{split}
\end{equation*}

For $ r > 0$ let $B(0, r) = \{\mu \in \measure \suchthat \variation{\mu} \leq r\}.$

We let $X_{\tau} = C\big([0,\tau]; B(0,2R)\big)$ and write $\norm{\mu}_{X_{\tau}} = \sup_{t\in[0,\tau]} \variation{\mu_t}.$

$X_{\tau}$ is a complete metric space with metric $d(\mu, \lambda) = \norm{\mu - \lambda}_{X_{\tau}}.$

\textbf{Step 1. Contraction mapping}
Consider the map $P$ defined on $\mu \in X_{\tau}$ by
\begin{equation*}
\begin{split}
\integrate{(P\mu)_t}{\test} = \integrate{\initialvalue}{\test} + \int_{0}^{t}
\Big(\dfrac{1}{2}\integrate{\mu_s \otimes \mu_s}{ K \opa{\test}} - \integrate{\mu_s}{\reacc_s \test + \opb{\test}} + \integrate{\lebesgue}{\coreacc_s \test}
\Big)\dd s
\end{split}
\end{equation*}
for $t\in[0,\tau]$, $\test\in \vanishing$.
If $\supnorm{\test} \leq 1,$ then we have $\quad\abs{\integrate{(P\mu)_t}{\test}} \leq$
\begin{equation*}
\begin{split}
&\leq \variation{\initialvalue} + \dfrac{3 \kbound}{2}\int_{0}^{t} \variation{\mu_s}^{2} \dd s + \big(\powerlaw \momentbound_{\beta} + 3 \upfrag\big) \int_{0}^{t} \variation{\mu_s} \dd s + 2 \upfrag \momentbound_{\gamma} \, t \\
&\leq \variation{\initialvalue} + \Big(\dfrac{3 \kbound}{2} \norm{\mu}_{X_{\tau}}^{2}  + \big(\powerlaw \momentbound_{\beta} + 3 \upfrag\big) \norm{\mu}_{X_{\tau}} + 2 \upfrag \momentbound_{\gamma}\Big) \tau \\
&\leq \variation{\initialvalue} + \big(6 \kbound R^2  + 2 \big(\powerlaw \momentbound_{\beta} + 3 \upfrag\big) R + 2 \upfrag \momentbound_{\gamma}\big) \tau \leq 2R.
\end{split}
\end{equation*}
whence $\variation{(P\mu)_t} \leq 2R$ and $P(X_{\tau}) \subset X_{\tau}.$
For $\mu, \nu \in X_{\tau}$
\begin{equation*}
\begin{split}
\abs{\integrate{(P\mu)_t - (P\nu)_t}{\test}} \leq  \dfrac{1}{2}\int_{0}^{t} \abs{\integrate{(\mu_s + \nu_s)\otimes (\mu_s - \nu_s)}{ K \opa{\test}}}\dd s + \\
+ \int_{0}^{t} \big|\integrate{\mu_s - \nu_s}{\abs{\reacc_s\test} + \abs{\opb{\test}}}\big| \dd s \leq \left(6 \kbound + \powerlaw \momentbound_{\beta} + 3 \upfrag \right) \tau \norm{\mu - \nu}_{X_{\tau}}
\end{split}
\end{equation*}
but $\left(6 \kbound + \powerlaw \momentbound_{\beta} + 3 \upfrag \right) \tau < 1$ by assumption, hence $P$ is a contraction mapping and must have a unique fixed point $\fixedpoint \in X_{\tau}.$

\textbf{Step 2. Nonnegativity of the solution}

To show that $\fixedpoint_t$ is actually in $\measurepos$ for each $t\in[0,\tau],$ we introduce the integrating factor
\begin{equation*}
\theta(t, x) = \exp{\left( \int_{0}^{t} \left[\reacc(s, x) + \int_{0}^{x} F(x - y, y) \dd y + \int_{(0,1]} K(x, y) \dd\mu_s(y) \right] \dd s \right)}
\end{equation*}
which is differentiable in $t$ with derivative
\begin{equation*}
\partial_t\theta(t, x) = \theta(t, x) \cdot \left(\reacc_t(x) + \int_{0}^{x} F(x - y, y) \dd y + \int_{(0,1]} K(x, y) \dd\mu_t(y) \right).
\end{equation*}

Both functions $\theta$ and $\partial_t \theta$ are Borel and bounded in $[0,\tau]\times (0,1]$.
We write $\theta'_{t}$ for the map $t \mapsto \partial_t \theta (t, \cdot).$

If we consider the measure $\fixedpos_t = \theta_t \fixedpoint \in \measure$ and write $S(\test)(x, y) = \overline{\test}(x + y),$ $\opd{\test}(x) = \int_{0}^{x} F(x - y, y) \test(y) \dd y,$ we see that $\fixedpos$ satisfies
\begin{equation*}
\begin{split}
&\integrate{\partial_t\fixedpos_{t}}{\test} = \integrate{\partial_t \fixedpoint_{t}}{\theta_{t} \test} + \integrate{\fixedpoint_{t}}{\theta'_{t} \test} = \\
&= \dfrac{1}{2}\integrate{\fixedpoint_t \otimes \fixedpoint_t}{ K S(\theta_t\test)} + \integrate{\mu_t}{\opd{\theta_t \test}} + \integrate{\lebesgue}{ \coreacc_t \theta_t \test}\\
&= \dfrac{1}{2}\integrate{\fixedpos_t \otimes \fixedpos_t}{ K S(\theta_t\test) (\theta_t^{-1} \otimes \theta_t^{-1})} + \integrate{\fixedpos_t}{\theta_t^{-1}\opd{\theta_t \test}} + \integrate{\lebesgue}{ \coreacc_t \theta_t \test}.
\end{split}
\end{equation*}

Let $\integrate{\mathcal{Q}_t(\nu)}{\test} = \dfrac{1}{2}\integrate{\nu \otimes \nu}{K S(\theta_t\test) (\theta_t^{-1} \otimes \theta_t^{-1})} +\integrate{\nu}{\theta_t^{-1}\opd{\theta_t \test}} + \integrate{\lebesgue}{ \coreacc_t \theta_t \test},$
then
\begin{equation*}
\integrate{\fixedpos_{t}}{\test} = \integrate{\initialvalue}{\test} +  \int_{0}^{t} \integrate{\mathcal{Q}_s(\fixedpos_s)}{\test} \dd s =: \integrate{(Q\fixedpos)_t}{\test}
\end{equation*}
and $\fixedpos$ is a fixed point of the map $Q.$

Using the bounds
\begin{equation*}
1 \leq \theta(t, x) \leq \exp{\tau\left( \powerlaw \momentbound_{\beta} +  2 \upfrag + 2 \kbound R \right)}
\end{equation*}
it is not hard to see that $Q$ is a contraction mapping in $C\big([0,\widetilde{\tau}]; B(0,L)\big)$ for some $L = L(R) > 0$ and a sufficiently small $\widetilde{\tau} = \widetilde{\tau}(R) \leq \tau.$
Thus $\fixedpos$ must be the unique fixed point in this space.

Consider the approximating sequence
\begin{equation*}
\fixedpos_t^{0} = \initialvalue, \quad \fixedpos_t^{n} = \initialvalue + \int_{0}^{t} \mathcal{Q}_s (\fixedpos_s^{n-1}) \dd s \quad \text{ for } n \geq 1.
\end{equation*}
Since $\initialvalue \in \measurepos,$ for nonnegative $\test \in \vanishing$ we have
$\integrate{\fixedpos_t^{n}}{\test} \geq 0$ for all $n \in \naturals,$ $t \in [0, \widetilde{\tau}].$
Taking the limit as $n \to \infty$ we see that
$\integrate{\fixedpos_t}{\test} \geq 0$ and hence $\fixedpos_t \in \measurepos$ for all $t \in [0,\widetilde{\tau}]$.
Since $\integrate{\fixedpoint_t}{\test} = \integrate{\fixedpos_t}{\theta_t^{-1}\test}$ and $\theta_t > 0$, we also have $\mu_t \in \measurepos$ for all $t\in[0,\widetilde{\tau}]$.

\textbf{Step 3. Extending the solution to all times}
Let $\test_1(x) = x > 0,$ then
\begin{equation*}
\opa{\test_1}(x, y) \leq 0, \quad \opb{\test_1}(x) = 0
\end{equation*}
and plugging $\test_1$ in the equation satisfied by $\fixedpoint,$ using that $\mu_t \in \measurepos$
\begin{equation}\label{useful_mass_bound}
\mass_{1}(t) = \integrate{\fixedpoint_t}{\test_1} \leq \integrate{\initialvalue}{\test_1} +  \int_{0}^{t} \integrate{\lebesgue}{\coreacc_s\test_1} \dd s \leq \variation{\initialvalue} + 2 \upfrag \momentbound_{\gamma} T
\end{equation}
for $t \in [0, \widetilde{\tau}].$
Let $\test_0(x) = \chi_{(0, 1]}(x),$
then $\opa{\test_0}(x, y) = - 1$ for $x, y \in (0,1]$
and
\begin{equation*}
-\opb{\test_0}(x) = \dfrac{1}{2}\int_{0}^{x} F(x - y, y) \dd y \leq \dfrac{\upfrag }{1 + \gamma} x^{1 + \gamma},
\end{equation*}
which combined with \eqref{useful_mass_bound} yields
\begin{equation*}
-\integrate{\fixedpoint_t}{\opb{\test_0}} \leq \dfrac{\upfrag }{1 + \gamma} \mass_{1 + \gamma}(t) \leq
\dfrac{\upfrag }{1 + \gamma}\big(\variation{\initialvalue} + 2 \upfrag \momentbound_{\gamma} T \big).
\end{equation*}
It follows that
\begin{equation*}
\begin{split}
&\variation{\fixedpoint_t} = \integrate{\fixedpoint_t}{\test_0} \leq \integrate{\initialvalue}{\test_0} - \int_{0}^{t} \integrate{\fixedpoint_s}{\opb{\test_0}} \dd s + \int_{0}^{t} \integrate{\lebesgue}{\coreacc_s\test_0} \dd s \leq\\
&\quad\leq \variation{\initialvalue} + \dfrac{\upfrag }{1 + \gamma}\big(\variation{\initialvalue} + 2 \upfrag \momentbound_{\gamma} T \big) T +  2 \upfrag \momentbound_{\gamma} \leq R
 \end{split}
\end{equation*}
for $t \in [0, \widetilde{\tau}].$

To conclude the reasoning, we note that in Step 1 we only used the bounds $\variation{\initialvalue} \leq R$ and the remaining estimates did not depend on the initial condition.
This way, redefining $P$ by letting $\initialvalue \to \fixedpoint_{\widetilde{\tau}},$ $g(t, x) \to g(\widetilde{\tau} + t, x),$ the fixed point equation $\mu = P\mu$ again has a unique solution in $X_{\widetilde{\tau}}$ which gives an extension of $\fixedpoint$ to the interval $[0, \min{\{2 \widetilde{\tau}, T\}}].$

By induction we see that the solution may be extended to $[0, T]$.
Since $T$ was arbitrary, we deduce that the constructed solution $\fixedpoint$ is in $\solspaceinfpos.$

The global uniqueness is now clear by the local uniqueness.
\end{proof}

From now on we assume that $K$ is a singular kernel satisfying \eqref{upper_bound_singular_kernel}.

\subsection{A priori estimates}

For each $j \in \naturals$ we consider the bounded coagulation kernel $\approxkernel = K \cdot \chi_{\{x, y > 1/j\}}.$
Each of the kernels has an associated $\reacc$-term $\approxreacc_t(x) = \int_1^{\infty} \approxkernel(x, y) g(t, y) \dd y$ and determines a unique solution $\approxsol \in \solspaceinfpos$ of the boundary valued coagulation fragmentation equation.

We derive some \emph{a priori} estimates for this sequence of solutions:
\begin{proposition}\label{a_priori_estimates_fragmentation}
Let $T > 0$ and assume $\momentbound_{\gamma} < \infty$.
Then there exists $C_T > 0$ depending on $\variation{\initialvalue}$ such that for all $\delta \in (0, 1)$ and $j \in \naturals$ we have
\[\sup_{t \in [0, T]}\variation{\approxsol_t} \leq C_T (1 + \momentbound_{\gamma}), \quad \sup_{t \in [0, T]} \approxsol_{t}(0,\delta) \leq \initialvalue(0, \delta) + C_{T}(1 + \momentbound_{\gamma}) \delta.\]
\end{proposition}
\begin{proof}
Let $\test_{0}(x) = \chi_{(0, 1]}(x).$
Computing, we have
\[\opa{\test_{0}}(x, y) = -1 \quad \text{ for } x, y \leq 1, \quad \opb{\test_{0}}(x) = - \dfrac{1}{2}\int_{0}^{x} F(x - y, y) \dd y \quad \text{ for } x \leq 1\]
and testing \eqref{a_solution} with $\test_{0},$ by \eqref{fragmentation_upper_bound}
\begin{equation*}
\begin{split}
\variation{\approxsol_t} \leq \variation{\initialvalue} + \upfrag \int_{0}^{t}  \variation{\approxsol_s} \dd s + 2 \upfrag \momentbound_{\gamma} t.
\end{split}
\end{equation*}
Applying Grönwall's inequality
\begin{equation*}
\sup_{t \in [0, T]}\variation{\approxsol_t} \leq \left(\variation{\initialvalue} + 2 \upfrag \momentbound_{\gamma} T \right) e^{\upfrag T},
\end{equation*}
which gives the first claim.

For $\delta \in (0, 1)$ let $\test_{\delta}(x) = \chi_{(0, \delta)}(x).$
We have
$\opa{\test_{\delta}}(x, y) \leq 0$ for $x, y \in (0,1]$ and
\begin{equation*}
\begin{split}
\opb{\test_{\delta}}(x) &= \dfrac{1}{2}\int_{0}^{x} F(x - y, y) \dd y \, \chi_{(0, \delta)}(x) -  \int_{0}^{x} F(x - y, y) \chi_{(0, \delta)}(y) \dd y \\
&\geq -\chi_{(0, \delta)}(x) \int_{0}^{x} F(x - y, y) \dd y  - \chi_{(\delta, 1]}(x) \int_{0}^{\delta} F(x - y, y) \dd y \\
&\geq - \dfrac{\upfrag}{1 + \gamma} \Big( 2 x^{1 + \gamma} \chi_{(0, \delta)}(x) + \chi_{(\delta, 1]}(x) \big(\delta^{1 + \gamma} + x^{1 + \gamma} - (x - \delta)^{1 + \gamma}\big)\Big) \\
&\geq - \dfrac{\upfrag}{1 + \gamma} \Big( 3\delta^{1 + \gamma} + (1 + \gamma) x^{\gamma} \delta\Big) \geq - \upfrag\dfrac{4 + \gamma}{1 + \gamma} \delta
\end{split}
\end{equation*}

Plugging $\test_{\delta}$ in \eqref{a_solution}
\begin{equation*}
\begin{split}
\approxsol_{t}(0,\delta) &\leq \initialvalue(0, \delta) - \int_{0}^{t} \integrate{\approxsol_s}{\opb{\test_{\delta}}} \dd s + \int_{0}^{t} \integrate{\lebesgue}{\coreacc_s \test_{\delta}} \dd s \\
&\leq \initialvalue(0, \delta) + \upfrag\dfrac{4 + \gamma}{1 + \gamma} \delta \int_{0}^{t} \variation{\mu_s} \dd s + 2 \upfrag \momentbound_{\gamma} \delta \, t
\end{split}
\end{equation*}
and the second claim is a consequence of the first.
\end{proof}

\begin{proposition}
\label{equicontinuity_frag}
Let $T > 0$ and assume $\momentbound_{\max{\{\beta, \gamma\}}} < \infty.$
Then there exists $C_T > 0$ depending on $\variation{\initialvalue}$ such that
\begin{equation}
\abs{\integrate{\approxsol_t - \approxsol_s}{\test}} \leq C_T \big( 1 + \momentbound_{\max{\{\beta, \gamma\}}} \big)^2 \big(\supnorm{\test} + \lipnorm{\test} \big)\abs{t - s}
\end{equation}
for all $\test \in \lip,$ $t, s \in [0, T]$ and $j \in \naturals.$
\end{proposition}
\begin{proof}
Without loss of generality $s < t.$
Substracting the equation for $\approxsol$ at times $t$ and $s$, applying Lemma \ref{uniform_bounds_operator} and Proposition \ref{a_priori_estimates_fragmentation} to the integrands
\begin{equation*}
\begin{split}
&\abs{\integrate{\approxsol_t - \approxsol_s}{\test}} \leq \\
&\leq \int_{s}^{t} \Big( \abs{\integrate{\approxsol_r}{\approxreacc_r \test}} +\abs{\integrate{\approxsol_r}{\opb{\test}}} + \dfrac{1}{2}\abs{\integrate{\approxsol_r \otimes \approxsol_r}{\approxkernel \opa{\test}}} + \abs{\integrate{\lebesgue}{\coreacc_r \test}} \Big)\dd r \\
& \leq \int_{s}^{t} \Big( \variation{\approxsol_r}4 \upperbound \massbdy_{\beta}(r)\lipnorm{\test} + \variation{\approxsol_r} 3 \upfrag \supnorm{\test}\Big) \dd r\\
 &\quad + \int_{s}^{t} \Big(  \variation{\approxsol_r}^2 4 \upperbound \lipnorm{\test} + 2 \upfrag \massbdy_{\gamma}(r)\supnorm{\test} \Big) \dd r\\
 &\leq C_T \Big(\big(1 + \momentbound_{\max{\{\beta, \gamma\}}}\big)^2 \lipnorm{\test} + \big(1 + \momentbound_{\max{\{\beta, \gamma\}}}\big) \supnorm{\test} \Big) |t - s|
\end{split}
\end{equation*}
for some $C_T > 0$ additionally depending on $\variation{\initialvalue}.$
\end{proof}

\subsection{Compactness}

The space $\vanishing$ is separable by functions in $\lip,$ more specifically
\begin{lemma}
\label{density}
There is a sequence $\{\test_n\}_{n\in \naturals} \subset \mathrm{Lip}_0(0,1]$ with $\norm{\test_n}_{\infty} \leq 2$ such that for every $\ep > 0$ and $f \in C_{0}(0,1]$ with $\norm{f}_{\infty} \leq 1$ we have $\norm{f - \test_k}_{\infty} < \ep$ for some $k\in\naturals.$
\end{lemma}

\begin{proof}
For each $n \in \naturals,$ choose $\{p_{n, j}\}_{j\in\naturals}$ an enumeration of the polynomials with rational coefficients restricted to $[1/n, 1].$
We extend
$p_{n, j}$ to $(0, 1]$ inside $\lip$ by linear interpolation
\begin{equation*}
\tilde{p}_{n,j} (x) =
\begin{cases}
p_{n, j}(x), & \text{ if } x\in [1/n, 1]\\
n \, p_{n, j}(1/n)\, x, & \text{ if } x\in (0, 1/n)
\end{cases}
\end{equation*}
and further
restrict to the $\tilde{p}_{n,j}$ such that $\supnorm{\tilde{p}_{n,j}} \leq 2.$
The fact that this family gives the desired sequence is a consequence of the Weierstrass' approximation theorem together with the definition of $f$ vanishing at zero.
\end{proof}

The next result is easy to check:

\begin{proposition}
\label{metrizability}
Let $B(0,1) = \{ \mu \colon \variation{\mu} \leq 1\} \subset \mathcal{M}_b(0,1]$
and let $\{\test_n\}_{n\in \naturals}\subset \mathrm{Lip}_0(0,1]$ be a sequence as in Lemma \ref{density}.
Then the topology generated by the norm
\begin{equation*}
\norm{\mu}_{\star} = \sum_{j = 1}^{\infty} \dfrac{\abs{\integrate{\mu}{\test_j}}}{2^j(1 + [\test_j]_{\normalfont\text{Lip}})},
\end{equation*}
restricted to $B(0,1)$ is equivalent to the weak-$\star$ topology in $B(0, 1).$
\end{proposition}

By the above, we can extract a weakly convergent subsequence from the approximate solutions:

\begin{proposition}\label{convergences}
There exists $\mu \in C\big([0,\infty);\weakmeasure\big)$ and a subsequence $\approxsolsub$ such that
\begin{itemize}
\item $\approxsolsub_t \weak \mu_t$ as $k \to \infty$ uniformly in $t\in [0,T]$ for all $T > 0$
\item $\approxsolsub_t \bweak \mu_t$ as $k \to \infty$ for each $t \geq 0$
\end{itemize}
\end{proposition}
\begin{proof}
By the first estimate in Proposition \ref{a_priori_estimates_fragmentation}, for each $t \geq 0$ the set $\{\approxsol_t\}_{j \in \naturals}$ is bounded with respect to the total variation norm and is therefore precompact in the weak-$\star$ topology.
Propositions \ref{equicontinuity_frag} and \ref{metrizability} show that the sequence is equicontinuous as a family of maps in the space $\weaksolspace.$
The first claim follows by iterative application of Arzelà-Ascoli to the sequence in $C\big([0,T_k]; \weakmeasure\big)$ for each $T_k = k \in \naturals$ and taking the diagonal sequence.
We do not relabel this convergent subsequence.

Let $t \geq 0.$
Given any subsequence $\{j_k\}_{k\in\naturals}$, the tightness estimate in Proposition \ref{a_priori_estimates_fragmentation} implies by Prokhorov's theorem that there exists a measure $\bar{\mu}_t\in \measure$ and
a further subsequence $\{j_{k_l}\}_{l\in\naturals}$
such that $\mu_t^{j_{k_l}} \bweak \bar{\mu}_t$ as $l \to \infty$.
In particular, $\mu_t^{j_{k_l}} \weak \bar{\mu}_t$ and by the uniqueness of the limit we actually have $\mu_t = \bar{\mu}_t.$
Since the subsequence was arbitrary, the second claim follows.
\end{proof}

\subsection{Passing to the limit}

For convenience we write $\approxsol$ for the convergent subsequence in Proposition \ref{convergences}.

\begin{lemma}
\label{zero_limits}
For every $t\in[0,T]$, $\test \in \lip$ we have
\begin{equation*}
\lim_{j \to \infty} \int_{(0,1]^2} \big( K(x,y) - \approxkernel(x, y)\big)\, \abs{\opa{\test}(x,y)} \dd \approxsol_t(x)\dd \approxsol_t(y) = 0,
\end{equation*}
\begin{equation*}
\lim_{j \to \infty} \int_{(0,1]} \big(\reacc_t(x) - \approxreacc_t(x)\big) \, \abs{\test(x)} \, \dd \mu^j_t(x) = 0.
\end{equation*}
\end{lemma}
\begin{proof}

By Lemma \ref{uniform_bounds_operator}
\[
\int_{(0,1/j]} \int_{(0, 1/j]}
K(x,y)\big|\opa{\test}(x,y)\big| \dd \approxsol_t(x)\dd \approxsol_t(y) \leq C \lipnorm{\test}
\approxsol_t(0, 1/j)^2,\]
\[
\int_{(0,1/j]} \int_{(1/j, 1]}
K(x,y)\abs{\opa{\test}(x,y)} \dd \approxsol_t(x)\dd \approxsol_t(y) \leq C \lipnorm{\test}
\approxsol_t(0, 1/j),\]
\[\int_{(0,1/j]} \reacc_t(x) \, \abs{\test(x)} \dd \approxsol_t(x) \leq C \momentbound_{\beta} \lipnorm{\test} \,
\approxsol_t(0, 1/j),\]
and by Proposition \ref{a_priori_estimates_fragmentation} we have $\approxsol_t(0, 1/j) \to 0$ as $j \to \infty$ since $\initialvalue \in \measurepos.$
\end{proof}

The results above allow us now to take the limit inside the equation.

\begin{proof}[Proof of Theorem \ref{existence_result}]
We claim that the measure $\mu$ constructed in Proposition \ref{convergences} is a solution of the boundary valued coagulation-fragmentation equation.

Let $\test \in \lip.$
Weak-$\star$ convergence implies
$\integrate{\approxsol_t}{\opb{\test}} \to \integrate{\mu_t}{\opb{\test}}$
uniformly in $t \in [0, T]$ for $T > 0,$ whence
\begin{equation*}
\int_{0}^{t} \integrate{\mu^j_s}{\opb{\test}} \dd s \to \int_{0}^{t} \integrate{\mu_s}{\opb{\test}} \dd s
\end{equation*}
for each $t \geq 0.$

We claim that for each $t \geq 0$ we have $\approxsol_t \otimes \approxsol_t \bweak \mu_t \otimes \mu_t$.
Indeed, by the \emph{a priori} estimates in Proposition \ref{a_priori_estimates_fragmentation} the sequence $\{\approxsol_t \otimes \approxsol_t\}_{j\in\naturals}$ is bounded in $\measureprod$ and tight.
By Prokhorov's theorem there is a further subsequence (which we do not relabel) and a measure $\lambda_t \in \mathcal{M}^{+}_b\big((0,1]^2\big)$ such that $\mu_t^j\otimes \mu_t^j \bweak \lambda_t$.
The subspace $\vanishing \otimes \vanishing \subset C_0\big((0,1]^2\big)$
of finite sums of products of functions
is uniformly dense by the Stone-Weierstrass theorem for locally compact spaces.
Moreover, for $\psi = \sum_{i = 1}^{N} a_i \otimes b_i$ with $a_i, b_i \in \vanishing$
the weak-$\star$ convergence implies that
\begin{equation*}
\integrate{\approxsol_t \otimes \approxsol_t}{\psi} = \sum_{i = 0}^{N} \integrate{\approxsol_t}{a_i} \integrate{\approxsol_t}{b_i} \to \sum_{i = 0}^{N} \integrate{\mu_t}{a_i} \integrate{\mu_t}{b_i} = \integrate{\mu_t\otimes \mu_t}{\psi}
\end{equation*}
as $j \to \infty,$ hence $\lambda_t$ and $\mu_t \otimes \mu_t$ coincide when tested against functions in $\vanishing \otimes \vanishing.$
By density they are also equal as elements of $\big(C_0\left((0,1]^2\right)\big)'$ and hence as Borel measures.

Lemma \ref{uniform_bounds_operator} shows that $\reacc_t \, \test \in \bounded$ and $K \opa{\test} \in C_b\big((0,1]\times(0,1]\big)$.
By the weak-$C_b$ convergence and Lemma \ref{zero_limits}
\begin{equation*}
\integrate{\approxsol_t}{\approxreacc_t \test} \to \integrate{\mu_t}{\reacc_t \test}, \quad \integrate{\approxsol_t \otimes \approxsol_t}{K_j \opa{\test}} \to \integrate{\mu_t \otimes \mu_t}{K \opa{\test}}
\end{equation*}
as $j \to \infty.$
Lemma \ref{uniform_bounds_operator} and Proposition \ref{a_priori_estimates_fragmentation} show that
\[|\integrate{\approxsol_t}{\approxreacc_t \test}| \leq \variation{\approxsol_t} \norm{\reacc_t \test}_{\infty} \leq C_T (1 + \momentbound_{\gamma})\momentbound_{\beta} \lipnorm{\test}\]
\[|\integrate{\approxsol_t \otimes \approxsol_t}{K_j \opa{\test}}| \leq \variation{\approxsol_t}^2 \norm{K \opa{\test}}_{\infty} \leq C_T(1 + \momentbound_{\gamma})^2 \lipnorm{\test}\]
independently of $j$ and by dominated convergence we may pass to the limit inside the integrals with respect to $\dd t,$ that is,
\[\int_{0}^{t} \integrate{\approxsol_s}{\approxreacc_s \test} \dd s \to \int_{0}^{t} \integrate{\mu_s}{\reacc_s \test} \dd s,
\]
\[ \int_{0}^{t} \integrate{\mu^j_s \otimes \mu^j_s}{ k_j \opa{\test}} \dd s \to \int_{0}^{t} \integrate{\mu_s \otimes \mu_s}{ k \opa{\test}} \dd s\]
for every $t \geq 0.$
\end{proof}
\begin{remark}
Notice that, by the above, $\mu$ is \emph{a priori} only in $C\big([0,\infty);\weakmeasure\big).$
In the proof of Theorem \ref{improvement_result} below we will see that it is actually in the space $\solspaceinfpos$ as required by our definition of solution (recall Definition \ref{definition_of_solution}).
\end{remark}

\subsection{Improvements}

With the additional assumption that we have lower bounds for the coagulation kernel $K,$ we obtain an upgraded version of \eqref{a_solution}, as stated in Theorem \ref{improvement_result}.

\begin{proof}[Proof of Theorem \ref{improvement_result}]
For $a \geq 0$ and $\delta > 0$ we let $\test_{a, \delta} \in \lip$ be piecewise defined by
\[\test_{a, \delta}(x) =
\begin{cases}
\delta^{-(1 + a)} x, & \text{ if } x < \delta,\\
x^{-a}, &\text{ if } x \geq \delta.
\end{cases}
\]
It is easy to check that $\opa{\test_{a, \delta}}(x, y) \leq 0$ and
writing $\test_{a, 0}(x) = x^{-a}$ we have
\begin{equation*}
\begin{split}
\test_{a, \delta}(x) \nearrow \test_{a, 0}(x), \quad -\opa{\test_{a, \delta}}(x, y) \nearrow -\opa{\test_{a, 0}}(x, y)
\end{split}
\end{equation*}
as $\varepsilon \to 0^{+}$ for $x, y \in (0,1].$
Moreover, if $\lambda < 1$
\begin{equation*}
\begin{split}
& -\opb{\test_{\lambda, 0}}(x) = \dfrac{1}{2}\int_{0}^{x} F(x - y, y) \Big((x - y)^{-\lambda} +  y^{-\lambda} - x^{-\lambda} \Big) \dd y \leq \\
& \leq \int_{0}^{x} F(x - y, y) y^{-\lambda} \dd y \leq 2 \upfrag \int_{0}^{x} y^{-\lambda} \dd y = \dfrac{2 \upfrag}{1 - \lambda} x^{1-\lambda},
\end{split}
\end{equation*}
and the estimates in Lemma \ref{uniform_bounds_operator} show that
\begin{equation*}
\begin{split}
& \integrate{\lebesgue}{\coreacc_t \test_{\lambda, 0}} \leq 2 \upfrag \massbdy_{\gamma}(t) \int_{0}^{1} x^{-\lambda} \dd x \leq \dfrac{2 \upfrag \momentbound_{\gamma}}{1 - \lambda}.
\end{split}
\end{equation*}

Plugging $\test_{0, \delta}$ in \eqref{a_solution} and letting $\delta \to 0$
\begin{equation*}
\begin{split}
&\variation{\mu_t} + \int_{0}^{t} \integrate{\mu_s}{\reacc_s} \dd s =\\
&\quad = \variation{\initialvalue} + \int_{0}^{t} \Big( \dfrac{1}{2}\integrate{\mu_s \otimes \mu_s}{\opa{\test_{0, 0}}} - \integrate{\mu_s}{\opb{\test_{0, 0}}} + \integrate{\lebesgue}{\coreacc_s \test_{0, 0}} \Big) \dd s \\
&\quad \leq \variation{\initialvalue} + C_T(1 + \momentbound_{\gamma}) < \infty
\end{split}
\end{equation*}
for $t \in [0, T].$
Since the lower bound \eqref{lower_bound_kernel} implies
\[\reacc(t, x) \geq \lowerbound \massbdy_{\beta}(t) x^{-\alpha}\]
we deduce that $\mass_{-\alpha}(t) < \infty$ for a.e. $t \geq 0.$

Let $\lambda \in (0, 1) \cap (0,\alpha]$ and let $t_0 > 0$ be such that $\mass_{-\lambda}(t_0) < \infty.$
For $t \in [0, T]$ with $t > t_0,$ if we substract the coagulation fragmentation equation at times $t$ and $t_0$, test with $\test_{\lambda, \delta}$ and then let $\delta \to 0$ we get
\begin{equation*}
\begin{split}
&\mass_{-\lambda}(t) + \lowerbound \int_{t_0}^{t} \mass_{-(\alpha + \lambda)}(s) \massbdy_{\beta}(s) \dd s \leq\\
&\quad\leq \mass_{-\lambda}(t_0) +  \int_{t_0}^{t} \Big( \integrate{\mu_s \otimes \mu_s}{\opa{\test_{\lambda, 0}}} - \integrate{\mu_s}{\opb{\test_{\lambda, 0}}} + \integrate{\lebesgue}{\coreacc_s \test_{\lambda, 0}} \Big) \dd s\\
&\quad \leq \mass_{-\lambda}(t_0) +  \int_{t_0}^{t} \Big( \dfrac{2 \upfrag}{1 - \lambda} \variation{\mu_s} + \dfrac{2 \upfrag \momentbound_\gamma}{1 - \lambda}\Big) \dd s\\
& \quad \leq \mass_{-\lambda}(t_0) + \dfrac{C_T}{1 - \gamma}(1 + \momentbound_{\gamma}) < \infty
\end{split}
\end{equation*}
where in the second inequality we have used that $\opa{\test_{\lambda, 0}} \leq 0,$ together with
the estimates at the beginning of this proof, and in the last inequality we have invoked the \emph{a priori} estimate in Proposition \ref{a_priori_estimates_fragmentation}.

It follows that $\mass_{-\lambda}(t) < \infty$ for $t > 0,$ $\mass_{-(\alpha + \lambda)} \massbdy_{\beta} \in L^{1}_{\text{loc}}[0, \infty)$ and in particular, $\mass_{-(\alpha + \lambda)}(t) < \infty$ for a.e. $t > 0.$
If $\alpha + \lambda < 1,$ applying the same reasoning to $\test_{(\alpha + \lambda), \delta}$ shows that $\mass_{-(\alpha + \lambda)}(t) < \infty$ for $t > 0$ and $\mass_{-(2\alpha + \lambda)}\massbdy_{\beta} \in L^{1}_{\text{loc}}[0, \infty).$
By iteration we obtain the first claim and moreover we have $\mass_{-(1 + \alpha - \varepsilon)}\in L^{1}_{\text{loc}}[0, \infty)$ for all $\varepsilon > 0.$
In particular, the \emph{a priori} singular terms in the equation define bounded measures as we explain next.

For $t \in [0, T],$ by \eqref{upper_bound_singular_kernel} and Proposition \ref{a_priori_estimates_fragmentation} we have
\[\begin{split}
&\abs{\integrate{\mu_t \otimes \mu_t}{K \opa{\test}}} \leq 3 \abs{\integrate{\mu_t \otimes \mu_t}{K}} \supnorm{\test} \leq 6 \upperbound \mass_{-\alpha}(t) \variation{\mu_t} \supnorm{\test} \leq \\
& \leq C_T(1 + \momentbound_{\gamma}) \mass_{-\alpha}(t) \supnorm{\test}
\end{split}\]
as well as
\[
\abs{ \integrate{\mu_t}{\reacc_t \test} } \leq 4 \upperbound \mass_{-\alpha}(t) \massbdy_{\beta}(t) \supnorm{\test} \leq 4 \upperbound \momentbound_{\beta} \mass_{-\alpha}(t)\supnorm{\test}
\]
where the moments $\mass_{-\alpha}$ are integrable in $[0, T]$ as explained in the paragraph above.

Writing
$\integrate{\int_{0}^{t}\opbstar{\mu_s} \dd s}{\test} = \int_{0}^{t} \integrate{\mu_s}{\opb{\test}} \dd s$ and so on,
by the arbitrariness of $\test \in \lip$ it follows that the total variations of the integrands are integrable in $[0, T]$ and for each $t \geq 0$ the measure
\[\int_{0}^{t}  \Big(\dfrac{1}{2}\opastar{K (\mu_s \otimes \mu_s)} - \reacc_s \mu_s - \opbstar{\mu_s} + \coreacc_s \lebesgue \Big) \dd s\]
is well defined as a Bochner integral with values in $\measure.$
We then have
\begin{equation}\label{measure_equality}
\mu_t = \initialvalue  + \int_{0}^{t}  \Big(\dfrac{1}{2}\opastar{K (\mu_s \otimes \mu_s)} - \reacc_s \mu_s - \opbstar{\mu_s} + \coreacc_s \lebesgue \Big) \dd s
\end{equation}
at the level of testing with functions in $\lip.$
We can further take test functions in $\vanishing$ as every measure appearing in \eqref{measure_equality} is bounded by the above.
This implies that \eqref{measure_equality} is actually an equality of elements in $\measure,$ and it must hold for even more general bounded functions in $\bddborel.$
Moreover, the identity implies that the solution $\mu$ is actually in $\solspaceinfpos$, as it is absolutely continuous.
\end{proof}

\section{Large time asymptotic behavior} \label{large_time_behavior}

\subsection{The case of coagulation}

In this subsection, we look at the particular case of a coagulating system without any fragmentation, i.e. $F = 0.$

\begin{proof}[Proof of Theorem \ref{moment_asymptotics_result}]
We will obtain moment estimates by applying the same reasoning as in the proof of Theorem \ref{improvement_result}.

For $\delta \geq 0,$ let $\test_{\gamma, \delta} \in \lip$ be piecewise defined as in the proof of Theorem \ref{improvement_result}.
Plugging $\test_{0, \delta}$ in equation \eqref{a_solution}, letting $\delta \to 0$ and by \eqref{lower_bound_kernel}
\begin{equation*}
\begin{split}
&\variation{\mu_t} + \lowerbound \int_{0}^{t} \mass_{-\alpha}(s) \massbdy_{\beta}(s) \dd s \leq\integrate{\mu_t}{\test_{0, 0}} + \int_{0}^{t} \integrate{\mu_s}{\reacc_s \test_{0, 0}} \dd s =\\
 &\quad = \integrate{\initialvalue}{\test_{0,0}} + \int_{0}^{t} \dfrac{1}{2}\integrate{\mu_s \otimes \mu_s}{ K \opa{\test_{0, 0}}} \dd s \leq \variation{\initialvalue},
\end{split}
\end{equation*}
whence $\mass_{-\alpha}\massbdy_{\beta} \in L^1[0, \infty).$
Proceeding as in the proof of Theorem \ref{improvement_result}, we deduce that actually $\mass_{-\alpha}(t) < \infty$ for $t > 0.$

Fix $t_0 > 0.$
Testing with $\test_{\alpha, \delta}$ and letting $\delta \to 0$
\begin{equation*}
\mass_{-\alpha}(t) + \lowerbound \int_{t_0}^{t} \mass_{-2\alpha}(s) \massbdy_{\beta}(s) \dd s \leq \mass_{-\alpha}(t_0)
\end{equation*}
whence $\mass_{-2\alpha} \massbdy \in L^1[0, \infty).$
Testing iteratively with $\test_{k\alpha, \delta},$ $k \in \naturals$ shows that $\mass_{-\lambda} \massbdy_{\beta} \in L^1[0, \infty)$ and $\mass_{-\lambda}(t) < \infty$ for all $t> 0,$ for every $\lambda \geq 0.$

It is then easy to show that $\mass_{-\alpha}$ is absolutely continuous and differentiable in $(0, \infty)$ with
\begin{equation*}
\mass_{-\alpha}'(t) + \integrate{\mu_t}{\reacc_t \test_{\alpha,0}} = \dfrac{1}{2}\integrate{\mu_t \otimes \mu_t}{ K \opa{\test_{\alpha, 0}}}.
\end{equation*}
Now since $\mass_{-\alpha}(t) \leq \mass_{-2\alpha}(t)$ and by the lower bound \eqref{lower_bound_kernel} we have
\[\begin{split}
&\mass_{-\alpha}'(t) + \lowerbound \mass_{-\alpha}(t) \massbdy_{\beta}(t) \leq \mass_{-\alpha}'(t) + \lowerbound \mass_{-2\alpha}(t) \massbdy_{\beta}(t) \leq\\
&\quad \leq \mass_{-\alpha}'(t) + \integrate{\mu_t}{\reacc_t \test_{\alpha,0}} = \dfrac{1}{2}\integrate{\mu_t \otimes \mu_t}{ K \opa{\test_{\alpha, 0}}} \leq 0
\end{split}\]
whence
\begin{equation*}
\mass_{-\alpha}(t) \leq \mass_{-\alpha}(t_0) \exp{\Big(-\lowerbound \int_{t_0}^{t}\massbdy(s) \dd s\Big)}
\end{equation*}
for $t \geq t_0 > 0.$

Testing inductively with $\test_{n\alpha, \delta}$ where $n \in \naturals$ and letting $\delta \to 0$ as above we deduce
\begin{equation*}
\mass_{-n\alpha}(t) \leq \mass_{-n\alpha}(t_0) \exp{\Big(-\lowerbound \int_{t_0}^{t}\massbdy(s) \dd s\Big)}
\end{equation*}
for $t \geq t_0 > 0$ and the claim follows.
\end{proof}

\subsection{Fragmentation and detailed balance}

Using a trick from \cite{stewart1989}, we can show that solutions of the coagulation fragmentation equation inherit the continuity of the initial datum when this one has a density with respect to the Lebesgue measure:

\begin{proposition}\label{ac_sols_fragmentation}
Under the hypotheses in Theorem \eqref{existence_result}, if $\initialvalue \ll \lebesgue$ then we also have $\mu_t\mres(0,1) \ll \lebesgue$ for all $t > 0$.
\end{proposition}

\begin{proof}
For $\delta > 0,$ $a \in \real$ let
\begin{equation*}
I_{a, \delta} = (a, a + \delta)\cap (0, 1), \quad \sizefunction(t) = \sup_{a \in \real} \mu_t(I_{a, \delta}).
\end{equation*}
Fix $\varepsilon > 0$ and let $\widetilde{\varepsilon} = \frac{1}{2} \varepsilon \exp{\big(-2\upperbound \int_{0}^{T} \mass_{-\alpha}(t) \dd t\big)}.$
Since $\initialvalue \ll \lebesgue,$ there exists $\delta > 0$ such that $\initialvalue(I_{a, \delta}) < \widetilde{\varepsilon}$ for all $a \in \real$.
We have for $x \in (0,1]$
\begin{equation*}
\begin{split}
-\mathcal{B}[\chifunction{a}](x) & = \dfrac{1}{2}\int_{0}^{x} F(x - y, y) \big(\chifunction{a}(y) + \chifunction{a}(x - y) - \chifunction{a}(x)\big) \dd y \\
& \leq  2 \upfrag \int_{0}^{x} \chifunction{a}(y) \dd y \leq 2 \upfrag \delta
\end{split}
\end{equation*}

Testing with $\test = \chifunction{a},$ for $t \in [0, T]$ we have
\begin{equation*}
\begin{split}
&\mu_t(I_{a,\delta}) + \int_{0}^{t} \int_{I_{a,\delta}} \big( \reacc_s(x) +
\integrate{\mu_s}{ K(x, \cdot)}\big) \dd \mu_s(x) \dd s = \\
& = \initialvalue(I_{a,\delta}) + \int_{0}^{t} \dfrac{1}{2}\int_{(0,1]^2} K(x, y) \chifunction{a}(x + y) \dd (\mu_s \otimes \mu_s) \dd s + \\
 & \quad + \int_{0}^{t} \Big( - \integrate{\mu_s}{\mathcal{B}[\chifunction{a}]} + \integrate{\lebesgue}{\coreacc_s  \chifunction{a}} \Big) \dd s \\
&\leq \initialvalue(I_{a,\delta}) + 2 \upperbound \int_{0}^{t} \mass_{-\alpha}(s) \sizefunction(s) \dd s + 2 \upfrag \, \delta \, \int_{0}^{t} \variation{\mu_s} \dd s + 2 \upfrag \, \delta \, \int_{0}^{t} \massbdy_{\gamma}(s) \dd s \\
&\leq \widetilde{\varepsilon} + 2 \upperbound \int_{0}^{t} \mass_{-\alpha}(s) \sizefunction(s) \dd s + C_T(1 + \momentbound_{\gamma}) \delta.
\end{split}
\end{equation*}
Ignoring the nonnegative terms in the left hand side, taking the supremum in $a \in \real$
\begin{equation*}
\sizefunction(t) \leq 2 \upperbound \int_{0}^{t} \mass_{-\alpha}(s) \sizefunction(s) \dd s + \widetilde{\varepsilon} + C_T(1 + \momentbound_{\gamma}) \delta,
\end{equation*}
and by Grönwall's inequality
\begin{equation*}
\sizefunction(t) \leq \Big(\widetilde{\varepsilon} + C_T(1 + \momentbound_{\gamma}) \delta \Big) \exp{\bigg( 2 \upperbound \int_{0}^{T}\mass_{-\alpha}(t) \dd t \bigg)}.
\end{equation*}
Choosing $\delta > 0$ small enough we find that $\sizefunction(t) < \varepsilon$
for $t \in [0, T],$ which implies $\mu_t \mres (0,1) \ll \lebesgue$ for all $t \geq 0.$
\end{proof}

If $\dd \initialvalue(x) = \initialvalueac(x) \dd x,$ we write $f(t, x)$ for the density with respect to Lebesgue measure of $\mu_t \mres(0, 1)$

\begin{corollary}\label{the_corollary}
There exists an absolutely continuous map $f \colon [0, \infty) \to L^{1}(0, 1)$
with a.e. derivative $\partial_t f \in L^1_{\text{loc}}\big([0, \infty); L^1(0,1)\big)$ such that
\begin{equation*}
\begin{split}
\partial_t f(t, x) &= \frac{1}{2} \int_{0}^{x} K(x - y, y) f(t, x-y) f(t, y)\dd y - \int_{0}^{1} K(x, y) f(t, x) f(t, y)\dd y \\
&\quad- \int_{1}^{\infty} K(x, y) f(t, x) g(t, y)\dd y + \int_{1 - x}^{\infty} F(x, y) g(t, x + y)\dd y\\
&\quad-\frac{1}{2} \int_{0}^{x} F(x - y, y) f(t, x)\dd y + \int_{0}^{1 - x} F(x, y) f(t, x + y)\dd y,
\end{split}
\end{equation*}
for a.e. $t > 0,$ $x \in (0, 1).$
\end{corollary}

\begin{remark}
$\mu_t$ will have a charge at $\{1\}$ in general, but the restriction to the open interval $(0, 1)$ is already a solution to the original equation \eqref{bdy_coagulation_fragmentation_equation} as Corollary \ref{the_corollary} shows.
The extension of $\mu_t$ to $(0,1]$ is a technical tool to ensure a nice compactness in the approximation of the singular problem (see Section \ref{existence}).
\end{remark}

Assume now that $g$ is time independent and satisfies the detailed balance condition (see Definition \eqref{bdy_detailed_balance}).
In this particular case, solutions of the boundary valued coagulation fragmentation equations minimize the \emph{entropy} functional
\begin{equation}
\entropy(t) = \int_{0}^{1} \Big(f(t, x) \Big[\log{\dfrac{f(t, x)}{Q(x)}} - 1\Big] + Q(x) \Big) \dd x.
\end{equation}

Some computations in this subsection are formal but can be rigorously justified by the same approximation argument in \cite[Section 5]{laurenccot2002continuous}.

\begin{proposition}\label{entropy_apriori_estimates}
We have
\begin{equation}\label{entropy_derivative}
\begin{split}
\entropy(0) \geq \entropy(t) &+ \int_{0}^{t} \int_{0}^{1}\int_{0}^{1 - x} \dfrac{1}{2} \energy{K(x, y) f(s, x) f(s, y), F(x, y) f(s, x + y)} \dd y \dd x \dd s \\
&+\int_{0}^{t} \int_{0}^{1}\int_{1 - x}^{1} K(x, y)\, \energy{f(s, x) f(s, y), Q(x) Q(y)} \dd y \dd x \dd s \\
&+\int_{0}^{t} \int_{0}^{1} \int_{1}^{\infty} K(x, y) Q(y) \,\energy{f(s, x), Q(x)} \dd y \dd x \dd s
\end{split}
\end{equation}
where
\begin{equation*}
\energy{a, b} =
\begin{cases}
(a - b) (\log{a} - \log{b}), &\text{ if } a, b > 0,\\
0, &\text{ if } a = b = 0,\\
+ \infty, &\text{ otherwise. }
\end{cases}
\end{equation*}
\end{proposition}
\begin{proof}
Differentiating under the integral sign
\[\dfrac{\dd \entropy(t)}{\dd t} = \int_{0}^{1} \partial_t f(t, x) \log{\dfrac{f(t, x)}{Q(x)}} \dd x.\]
\eqref{entropy_derivative} then follows by testing the coagulation fragmentation equation with $\test(t, x) = \log{\dfrac{f(t, x)}{Q(x)}},$ using that $K(x, y) Q(x)Q(y)  = F(x, y) Q(x + y)$ and integrating in time.
\end{proof}

\begin{lemma}\label{useful_lemma}
Let $E \subset (0, 1)$ be Lebesgue measurable.
Then for each $N \geq e^2$ there is $C_N > 0$ such that
\begin{equation}
\int_{E} f(t, x) \dd x \leq C_N \int_{E} Q(x) \dd x + \dfrac{2}{\log{N}} \entropy(t)
\end{equation}
\end{lemma}
\begin{proof}
This is Lemma 3.1. in \cite{laurenccot2002continuous} for the particular case $\eta = \chi_{E}$.
\end{proof}

\begin{lemma}\label{a_priori_asymptotics}
There is a constant $C > 0$ such that
\begin{equation}
\int_{0}^{1} f(t, x) \dd x + \int_{0}^{1} f(t, x)\abs{\log{\dfrac{f(t, x)}{Q(x)}}} \dd x \leq C
\end{equation}
for all $t \geq 0.$
\end{lemma}
\begin{proof}
This is Lemma 3.2. in \cite{laurenccot2002continuous}.
\end{proof}

\begin{proof}[Proof of Theorem \ref{asymptotics_detailed_balance_result}]
Fix $T > 0$ and let $\{t_n\}_{n \in \naturals} \subset (0, \infty)$ be
a sequence such that $t_n \to \infty$ as $n \to \infty.$
Consider the functions $f_n(t, x) = f(t_n + t, x)$ for $t \in [0, T],$ $x \in (0, 1).$
We will extract a weakly convergent subsequence from $\{f_n\}$ and show that the limit is the equilibrium determined by the detailed balance.

\textbf{Step 1. Precompactness}

Let $\varepsilon > 0$ and $E \subset (0, 1)$ be Lebesgue measurable.
Choosing $N$ large enough and using that $Q \in L^1(0, 1),$ Lemma \ref{useful_lemma} implies that $\int_{E} f(t, x) \dd x < \varepsilon$ if $\lebesgue(E) < \delta$ for some $\delta$ small enough.
The family $\{f_n(t)\}_{n \in \naturals}$ is then equiintegrable for each $t \geq 0.$

By Lemma \ref{a_priori_asymptotics} $\{f_n(t)\}_{n\in\naturals}$ is bounded in $L^1(0, 1).$
The Dunford-Pettis theorem now implies that $\{f_n(t)\}_{n \in \naturals}$ is precompact in the weak $L^1(0, 1)$ topology.

By the estimates for the contraction mapping in the proof of Proposition \ref{wp_bdd_frag} together with the \emph{a priori} bounds in Lemma \ref{a_priori_asymptotics}, there exists a $C > 0$ such that
\begin{equation*}
\abs{\integrate{f_{n}(t) - f_{n}(s)}{\test}} \leq C(1 + \momentbound_{\max{\{\beta, \gamma\}}}) \abs{t - s} \supnorm{\test}
\end{equation*}
for all $t, s \geq 0,$ $\test \in L^{\infty}(0,1).$
It follows that the sequence $\{f_n\}_{n\in\naturals}$ is strongly equicontinuous into $L^1(0, 1)$.

By Arzelà-Ascoli there exists a map $\bar{f} \colon [0, T] \to L^1(0, 1)$ which is continuous with respect to the weak topology in $L^1(0, 1)$ and a subsequence (not relabelled) such that
$f_n(t, \cdot) \to \bar{f}(t, \cdot)$ in the weak topology, uniformly in $t \in [0, T].$

\textbf{Step 2. Stability}

Looking at the last term
in inequality \eqref{entropy_derivative} we notice that
\begin{equation*}
\begin{split}
&\int_{0}^{T}  \int_{0}^{1}\int_{1}^{\infty} K(x, y) Q(y) \,\energy{f_n(t, x), Q(x)} \dd y \dd x \dd t =\\
&\quad = \int_{t_n}^{t_n + T} \int_{0}^{1}\int_{1}^{\infty} K(x, y) Q(y) \,\energy{f(t, x), Q(x)} \dd y \dd x \dd t,
\end{split}
\end{equation*}
which is an expression converging to zero as $n \to \infty.$
Indeed, letting $t \to \infty$ in the \emph{a priori} estimates in Proposition \ref{entropy_apriori_estimates} we see that the time integrals are finite, implying that their tails converge to zero, hence the claim as $t_n \to \infty$ when $n \to \infty.$

Using the inequality
\begin{equation*}
x \leq N y + \dfrac{1}{\log{N}} (x - y) (\log{x} - \log{y})
\end{equation*}
with $N = 1 + \varepsilon$ we have
\[f_n(t,x) - Q(x) \leq \varepsilon Q(x) + \dfrac{1}{\log{(1+\varepsilon)}}\energy{f_n(t, x), Q(x)} \]
\[ Q(x) - f_n(t, x) \leq \varepsilon f_n(t, x) + \dfrac{1}{\log{(1+\varepsilon)}}\energy{f_n(t, x), Q(x)}.\]
By the above, applying Lemma \ref{a_priori_asymptotics} and recalling that $\reacc(x) \leq \powerlaw \momentbound_{\beta}$ for $x \geq 1$ (since $K$ is a \emph{bounded coagulation kernel} according to definition \ref{kernel}) we find that
\begin{equation*}
\begin{split}
&\quad \int_{0}^{T}  \int_{0}^{1} \reacc(x) \abs{f_n(t, x) - Q(x)} \dd x \dd t \leq \\
&\leq  \varepsilon\powerlaw \momentbound_{\beta} T \|Q\|_{L^1} + \varepsilon\powerlaw \momentbound_{\beta} \int_{0}^{T} \int_{0}^{1} f_n(t, x) \dd x \dd t \\
&\quad +\dfrac{1}{\log{(1 + \varepsilon)}} \int_{0}^{T}  \int_{0}^{1}\int_{1}^{\infty} K(x, y) Q(y) \,\energy{f_n(t, x), Q(x)} \dd y \dd x \dd t\\
&\leq C \varepsilon + \dfrac{1}{\log{(1 + \varepsilon)}} \int_{0}^{T} \int_{0}^{1}\int_{1}^{\infty} K(x, y) Q(y) \,\energy{f_n(t, x), Q(x)} \dd y \dd x \dd t
\end{split}
\end{equation*}
for some $C.$
Then
\begin{equation*}
\limsup_{n \to \infty}\int_{0}^{T}  \int_{0}^{1} \reacc(x) \abs{f_n(t, x) - Q(x)} \dd x \dd t \leq  C \varepsilon
\end{equation*}
and since $\reacc(x) > 0$ for all $x \in (0, 1)$ by assumption, we conclude that
\begin{equation}
f_n(t, x) \to Q(x) \text{ for a.e. } t \in [0, T], x \in (0, 1).
\end{equation}
By the uniqueness of the weak limit
$\bar{f} = Q|_{(0, 1)} = f_{\infty}$ and in particular for every sequence with $t_n \to \infty$ we have
$f(t_n) \to f_{\infty}$ in the weak topology, hence the claim.
\end{proof}

\begin{remark}
Notice that in the proof above we have only used the bound of one of the terms in the \emph{a priori} estimates for the entropy.
This is enough to deduce the convergence of $f$ to the equilibrium,
but it might be relevant to check whether the entropy dissipation estimate can be used to obtain information about the rate of convergence to equilibrium of the solution.
\end{remark}

\section{Conclusions}\label{conclusions}

In this paper we have studied boundary value problems for coagulation -fragmentation equations 
which, in spite of the fact that they are of interest in some problems of atmospheric science,
have not yet been considered in the mathematical literature.

We have successfully established the existence of solutions
for the singular kernels considered in the applications
under the sole assumption of boundedness in time for the moments of the boundary datum.

In the case of coagulating systems, we have shown that the size distribution function for small clusters decreases in time at a rate which depends on the moments of the boundary size distribution function.
This result agrees with what is expected under such conditions, as the small clusters can only merge with each other to give larger ones until there are no small particles left.

On the other hand, the addition of a fragmentation mechanism complicates the analysis of the large time asymptotics.
Under the assumptions on the boundary datum in Theorem \ref{asymptotics_detailed_balance_result} (namely, it being time independent and satisfying a detailed balance condition) we have seen that
there exists a decreasing Lyapunov functional implying the convergence of solutions to a unique equilibrium.

However, it is not clear whether convergence to an equilibrium under weaker hypotheses is possible.
The usual detailed balance appearing in the classical problems without boundary involves the coagulation and fragmentation kernels only, but the introduction of boundary conditions
requires to complement this condition with additional assumptions on the boundary data.
This is not surprising, since the coagulation-fragmentation model with nonequilibrium data (even when these are time independent) can be thought of as an open system in contact with reservoirs and it is well known that open chemical systems out of equilibrium might yield an oscillatory behavior \cite{field1974oscillations}.

\section*{Acknowledgments}
The author thanks Juan J. L. Vel\'{a}zquez for the suggestion of the problem as well as for many helpful remarks and conversations during the elaboration of this work.
The author acknowledges financial support from the Spanish Ministry of Economy and Competitiveness (MINECO)
through the Mar\'{i}a de Maeztu Program for Units of Excellence in R\&D (MDM-2014-0445-18-1),
is supported by MINECO grant MTM2017-84214-C2-1-P,
is a member of the Barcelona Graduate School of Mathematics (BGSMath)
and is part of the Catalan research group 2017 SGR 01392.



\end{document}